\documentclass[a4paper,reqno]{amsart}

\usepackage{amsmath}

\usepackage{amsfonts}

\usepackage{amssymb}

\usepackage{verbatim}

\usepackage{enumerate}

\newcommand{\calP}{{\mathcal P}}

\newcommand{\calX}{{\mathcal X}}

\newcommand{\calY}{{\mathcal Y}}

\newcommand{\calZ}{{\mathcal Z}}

\newcommand{\sfS}{{\mathsf S}}

\newcommand{\calQ}{{\mathcal Q}}

\newcommand{\calD}{{\mathcal D}}

\newcommand{\mG}{{\mathbb G}}

\newcommand{\mV}{{\mathbb V}}

\newcommand{\mE}{{\mathbb E}}

\newcommand{\mL}{{\mathbb L}}

\newcommand{\ppp}{{p}}

\newcommand{\gra}{{\alpha}}

\newcommand{\gre}{{\varepsilon}}

\newcommand{\grb}{{\beta}}

\newcommand{\grg}{{\gamma}}

\newcommand{\grf}{{\varphi}}

\newcommand{\vuoto}{{\varnothing}}

\newcommand{\supp}[1]{{supp\,(#1)}}

\newcommand{\calF}{{\mathcal F}}

\newcommand{\calA}{{\mathcal A}}

\newcommand{\calC}{{\mathcal C}}

\newcommand{\calR}{{\mathcal R}}

\newcommand{\calU}{{\mathcal U}}

\newcommand{\calW}{{\mathcal W}}

\newcommand{\mZ}{{\mathbb Z}}

\newcommand{\mR}{{\mathbb R}}

\newcommand{\st}{:}

\newcommand{\mand}{{\text{ and }}}

\newcommand{\grs}{\sigma}

\DeclareMathOperator{\codim}  {codim}

\DeclareMathOperator{\card} {card}

\DeclareMathOperator{\Int} {Int}

\DeclareMathOperator{\Est} {Est}

\newcommand{\lra} {\longrightarrow}

\newcommand{\senza}{\setminus}

\newtheorem{theorem}{Theorem}

\newtheorem{lemma}[theorem]{Lemma}

\newtheorem{prop}[theorem]{Proposition}

\newtheorem{corollary}[theorem]{Corollary}

\title[Perfect simulation for the random cluster model]{Perfect simulation for
 the infinite random cluster model, Ising and Potts models at low or high temperature}

\author{Emilio De Santis, Andrea Maffei}

\begin{document}

\begin{abstract}
In this article we create a new algorithm for the perfect simulation
of the infinite Potts model at a sufficiently small or at a
sufficiently high temperature, in particular under the transition
phase temperature. We study the model for free boundary conditions
and we give some consequences for the constant boundary conditions.
\end{abstract}

\maketitle

Given a countable graph $ G = (V, E) $, a positive number $q$ and
parameters $p= \{ p_e \in [0,1]: e \in E\} $, the random cluster
measure is defined on the measurable space $(\Omega,\mathcal{F}) $,
where $ \Omega=\{ 0,1\}^E$ and $\calF$ is the $\sigma$--algebra
generated by finite cylinders. This measure was introduced by
Fortuin and Kasteleyn as a way to study the Ising and Potts models
(see \cite{FK}).

In our paper we do not require the parameters $p_e$ to be all equal
to the same constant; however we limit our attention only on the
models with $q>1$. This choice has been made both to maintain the
article simpler and also because the case $q > 1$ has important
connections with the statistical mechanics and in particular with
the models of Ising and Potts.

The aim of this paper is to construct an algorithm which gives a
perfect simulation of a random cluster measure on a finite region of
an infinite graph. Notice that even if the perfect simulation is
obtained only on a finite region, it takes into account the fact
that the random field on this region is influenced by the value of
the filed on the whole infinite graph. On any finite graph it is
possible to define, in an explicit and computable way, the random
cluster measure. However, the random cluster measure associated to
the finite region in this way, in general, is not the restriction of
the measure of the infinite graph on this finite region.

Now we briefly explain how this simulation is obtained. As will be
recalled in Sections \ref{sec:misure} and \ref{sec:dinamica} the
random cluster measure is invariant under a markovian dynamics.
Introduce a countable number of copies of the graph $G$ and think
them as placed at level, $0$, $-1$, $-2$, etc. Choose also an order
of the edges of $G$: $e_1$, $e_2$, etc. For a configuration
$\omega\in \Omega$ of the graph at level $-M$ create new
configurations at level above $-M$ updating the value of
$\omega_{e_k}$ one at the time, according to the conditional
probabilities that depends geometrical shape of the configuration.
The details of this dynamics are given in Section
\ref{sec:dinamica}. Let us just recall here that the law used to
update the value of $\omega_{e_k}$ depends on the existence of a
connected path of edges $e$, different from $e_k$, such that
$\omega_{e}=1$, joining the end vertexes of $e_k$. The construction
of this dynamics can be seen as a particular case of the Glauber
dynamics. In the study of the random cluster measures similar
dynamics where already considered, for example by Grimmett in
\cite{Grimmett95}.

We now construct a coupling of all the dynamics for each possible
initial configuration. We color with black, gray or white, in an
independent way, the edges of all the copies of $G$. The law used to
color the edge $e$ at level $-\ell$ depends only on the parameter
$p_e$. The coupling is constructed in such a way that for all
possible initial configurations $\omega$, the value of the
configuration in $e$ at level $-\ell$ will be $1$ whenever the edge
$e$ at level $-\ell$ is black and $0$ whenever it is white.

Now fix a finite region $F$ of $G$. In Theorem \ref{teo:pqalta} in
the case of high temperature (which corresponds to $p_e$ close to
zero), we prove that, for almost all coloring as above, there exists
a finite region $C_F^b$ of the union of all the copies of $G$ that
is ``surrounded'' by white edges and containing the region $F$ at
level $0$. Finally, in Theorem \ref{teorema2} we prove that, given a
coloring, if the region $C^b_F$ is finite then we can determine a
bigger finite region $\bar H$ such that the output of the dynamics
described above at level $0$ in the region $F$ does not depend on
the choice of $\omega$ and on the coloring outside the region $\bar
H$.

Similar results are proved in Theorems \ref{teo:pqbassa} and
\ref{teorema1} in the case of low temperature which corresponds to
$p_e$ close to one. However in this case the meaning of the word
``surrounded'' is different. To treat this case we have to make some
further assumption on the geometry of the graph $G$. For this reason
in Section \ref{sec:grafisimpliciali} we introduce the concept of
simplicial graph. These are the graphs that can be obtained as the
vertexes and edges of a tessellation of an euclidean space. The
first example we have in mind is the cubic lattice $\mathbb{L}^d=(
\mathbb{Z}^d, \mathbb{E}_d)$. In Section \ref{sec:grafisimpliciali}
we prove also the results required in the proof of Theorems
\ref{teo:pqbassa} and \ref{teorema1}.

In Section \ref{sec:simulazione} we summarize these results and we
show how to obtain an algorithm for the perfect simulation of the
random cluster measure. We give the details only in the case of low
temperature which is the most difficult and interesting. The case of
high temperature can be treated in a similar way. As a byproduct we
also obtain the uniqueness of the random cluster measure at low or
high temperature, which, al least in the case of $\mathbb L^d$ is
well known (see \cite{Grimmett95,Grimmett}).

In the last section we analyze the connection between the perfect
simulation of the random cluster measure and the Potts model. We
construct  an algorithm to simulate  the Potts model with free
boundary conditions at low or high temperatures.







We also refer to some literatures on perfect simulation. In
\cite{PW} was introduced the perfect simulation algorithm for Markov
chains. If the chain is ergodic with this algorithm it is possible
to simulate the unique stationary measure associated with the chain.
This paper has started some new research fields. One area of
research concerns the Markov fields (see \cite{HS00,DP08}); a second
one concerns the processes with infinite memory (see
\cite{CFF,DP2,Gallo}). Recently, these two areas of research have
been in some sense unified by studying Gibbs measures with infinite
interaction range (see \cite{GLO,DLiss}). Our paper is included in
the latter context.









\section{Some notations on graphs} \label{notation}

In this section we recall some definitions on graphs that will
be used in the sequel. 

In this paper a graph will be a collection of two sets, $V$ called
the set of vertexes and $E$ called the set of edges, and of a map
from $E$ to the set of unordered pairs of different elements of $V$.
The pair associated to an edge $e$ are called the end vertexes of
$e$ and the two vertexes are said to be \emph{adjacent}. As it is
common in the literature we will denote a graph by $(V,E)$. A
\emph{path} in $G$ joining the vertexes $u$ and $v$ is a sequence
$e_1,\dots,e_m$ in $E$ such that $e_i$ and $e_{i+1}$ have a common
vertex, $u$ is an end vertex of $e_1$ and $v$ is an end vertex of
$e_m$. The integer $m$ is called the length of the path. Two
vertexes are said to be in the same connected component if there is
a path joining them. The graph-distance of two vertexes $u$ and $v$,
denote with $\delta_G (u,v)$, is the length of a minimal path of
joining them, and it is infinite if the two vertexes are in
different connected components. We denote by $B_G(v,r)$ the ball of
center $v$ and ray $r$ with respect to this distance.







\section{The random cluster measure}\label{sec:misure}

In this section we define the random cluster measure introduced by
Fortuin and Kasteleyn as explained in the book of Grimmett
\cite{Grimmett}. Since our setting will be slightly more general
than the one exposed by Grimmett we give the construction of the
measure. However all the arguments given in his book easily
generalize to our setting, so we refer to \cite{Grimmett}, Chapter
4, for the details.

\subsection{Construction as thermodynamic limit}

For our constructions we fix a graph $G=(V,E)$. We further assume
that it is countable of finite degree, meaning that $V$ is countable
and that every vertex is an end vertex of a finite number of edges.

Set $\Omega=\{0,1\}^E$ and let $\mathcal F$ be the $\sigma$-algebra
generated by finite cylinders. If $\omega \in \Omega$ we denote by
$E(\omega)$ the set of the elements $e \in E$ such that $\omega_e
=1$. To define a random cluster measure we also fix parameters $\ppp
=(p_e \in [0,1]: e\in E)$, and $q\in (0,\infty)$. For simplicity in
this paper we will assume $q\geq 1$ which is the more significant
for the application to statistical mechanics.

There are two ways of defining random cluster measure on $G$. The
first method is as limit on finite subgraph and is called the
thermodynamic limit. The second method is by giving the conditional
probabilities on finite subgraph and it is called the
Dobrushin-Lanford-Ruelle or DLR method. We now explain briefly the
construction as thermodynamic limit.

Given $\xi \in \Omega$ and $F\subset E$ a finite set, let
$\Omega^{\xi}_{F}= \{\omega \in \Omega : \omega_e=\xi_e$ for all $e
\not\in F \}$. We define the measure $\phi^\xi _{F,\ppp,q}$ on
$\Omega$ by:
\begin{equation}\label{FKmeasure} \phi^\xi _{F,\ppp,q}(\omega) =
\begin{cases}
\frac{1}{Z_{\xi,F}} \left [ \prod_{e
\in F} p_e^{\omega_e} (1-p_e)^{1- \omega_e} \right ]q^{ k(\omega,F)} &\text{if }
\omega\in \Omega^\xi_F, \\
0  &\text{otherwise},
\end{cases}
\end{equation}
where $k(\omega,F)$ is the number of connected components of the
graph $(V,E(\omega))$ that intersects $F$ and $Z_{\xi,F}$ is just
the normalizing constant.

Following \cite{Grimmett} Definition 4.15 we say that a probability
$\phi$ on $(\Omega,\mathcal F)$ is a \emph{limit random cluster
measure} if there is a sequence $(\xi_n,F_n)$ such that $\phi$ is
the weak limit of the measures $\phi^{\xi_n}_{F_n,\ppp,q}$ and we
denote by $\mathcal W_{\ppp,q}$ the set of these measures.

If we fix $\xi_n$ to be constantly equal to $1$ (resp.\ to $0$) the
limit of the measures $\phi^{\xi_n}_{F_n,\ppp,q}$ exists, and it
does not depend on the choice of the sequence $F_n$ (see
Theorem~4.19 in \cite{Grimmett}). This limit will be denoted by
$\phi_{\ppp,q}^1$ (resp.\ $\phi_{\ppp,q}^0$). Moreover, for all
$\phi \in \mathcal W_{\ppp,q}$, we have
$$
\phi_{\ppp,q}^0\leq_{st} \phi\leq_{st} \phi_{\ppp,q}^1.
$$
where $\leq_{st}$ is the \emph{standard stochastic order} as in
\cite{Grimmett} Section 2.1.

Another important property of these measures is the so called finite
energy property. Let
\begin{equation}\label{eq:phat}
\hat p_e = p_e/(p_e+q(1-p_e)), \end{equation} then for all $\phi
\in\mathcal W _{\ppp,q}$ and $e\in E$, we have
$$
\hat p_e \leq \phi(L_e|\mathcal T _e) (\omega)\leq p_e,\qquad a.s.
$$
where $L_e =\{\omega: \omega_e=1\}$, $\mathcal T_e$ is the
$\sigma$-algebra generated by the finite cylinders with base
contained in $E\setminus\{e\}$.

\subsection{DLR construction}

In the case of a finite graph $G=(V,E)$ the definition given above
furnishes a unique measure on $(\Omega,\mathcal F)$ which is
characterized by the conditional probability of $\omega_e=1$ given
the values of $\omega$ in $E\setminus\{e\}$. In this model these
probabilities depend on the existence of a path in
$E(\omega)\setminus\{e\}$ joining the two end vertexes of the edge
$e$.  Let $K_e$ be the set of configurations $\omega$ having this
property.

In the case of infinite graph this property can be formalized as
follows:
\begin{equation}\label{DLR}
\phi(L_e|\mathcal T _e)(\omega) =
\begin{cases}
p_e &\text{if } \omega \in K_e, \\
\hat p_e  &\text{if } \omega \not \in K_e.
\end{cases}
\end{equation}
A probability $\phi$ on $(\Omega,\mathcal F)$ is a \emph{DLR random
cluster measure} if it satisfies equation \eqref{DLR} for all $e\in
E$. We denote by $\mathcal R_{\ppp,q}$ the set of these measures.

In the infinite setting the two definitions of random cluster
measure are not always equivalent. However (see \cite{Grimmett},
Chapter 4, Section 4) it is known that $\mathcal R_{\ppp,q}$ is not
empty, in particular $\phi^0_{\ppp,q}$ and $\phi^1_{\ppp,q}$ are
elements of $\mathcal R_{\ppp,q}$ and for all $\phi \in \mathcal
R_{\ppp,q}$ one has
$$
\phi_{\ppp,q}^0\leq_{st} \phi\leq_{st} \phi_{\ppp,q}^1.
$$
In particular notice that
$$
\card(\mathcal W_{\ppp,q})=\card(\mathcal R_{\ppp,q})=1 \quad \text{
if and only if } \quad \phi_{\ppp,q}^0=\phi_{\ppp,q}^1.
$$



\section{Construction of the dynamics}\label{sec:dinamica}

In this paper, following the literature, to give a sample of a
measure on $(\Omega,\calF)$ we introduce families of auxiliary
random variables $(u_e)_{e\in E}$ that are independent and uniformly
distributed on $[0,1]$, which in the algorithm we present in
Sections \ref{sec:simulazione} and \ref{sec:applicazioni} can be
thought as the output of a pseudorandom function on a computer.

We now define a stochastic dynamic such that the measures in
$\mathcal R_{\ppp,q}$ are invariant. A similar dynamics was already
considered by Grimmett in \cite{Grimmett95}.

Let $G=(V,E)$ be a countable graph of finite degree and we choose an
order for its edges so that $E=\{ e_1, e_2 , \ldots\}$.











For a negative number $N$ we define $\calA_N = \{(n,k)\in \mZ \times
\mZ: -1\geq n\geq N$ and $k\geq 1\}$ (resp.\ $\calA=\bigcup_N
\calA_N$) and $\calU = [0,1]^{\calA}$.  On $\calU$  we put the
Lebesgue product measure so that the coordinates $u_{n,k}$ are
i.i.d.\ random variables having uniform distribution on $[0,1]$. We
define also $\tilde \calA_N = \calA_N \cup\{(n,0)\st n=N,\dots,0\}$
and $\tilde \calA = \bigcup_N \tilde \calA_N$.





For a fixed $N<0$ and for a fixed $X_{N,0}:\calU \lra \Omega$ let
$(X_{n,k}:\calU \longrightarrow \Omega$ such that $(n,k) \in \tilde
\calA_N)$ be a process with values in $\Omega$ constructed in the
following way: given $(n,k)\in \tilde \calA _N$
\begin{equation}\label{eqdefX}
(X_{n,k+1}(u))_e =
\begin{cases}
(X_{n,k}(u))_e &\text{if } e \neq e_{k+1}, \\
1 & \text{if } u_{n,k+1} < p_e,\; e = e_{k+1}\text{ and } X_{n,k}(u)
\in K_e, \\
1 & \text{if }u_{n,k+1} < \hat p_e,\; e = e_{k+1}\text{ and }
X_{n,k}(u)
\notin K_e, \\
0 & \text{otherwise.}
\end{cases}
\end{equation}
Furthermore notice that for all $n$ there exists the limit $X_{n,k}$
for $k$ going to infinity. We construct $X_{n+1,0}$ as this limit.
We call such a process an $FK^N_{\ppp,q}$-process.

All DLR random cluster measures are invariant under this process.
Indeed if $\phi \in \calR_{\ppp,q}$ and $X_{n,k}$ has law $\phi$
then $X_{n,k+1}$ has the same law. Hence, for $h > k$ and $m=n$, or
for $m>n$, $X_{m,h}$ has law $\phi$.




Finally, for any integer $N<0$ and any $\omega\in \Omega$ we denote
with $X_{n,k}^{(\omega,N)}$ the $FK^N_{\ppp,q}$-process constructed
starting with $X_{N,0}(u)=\omega$ for $u\in \calU$.

\medskip

In the main result of this paper, under suitable assumption on the
parameters $\ppp ,q$ and on the graph $G$, given a finite $F \subset
E$ and $u\in \calU$ we will show how to determine a integer $N$ such
that $(X^{(\omega,N)}_{0,0}(u))_e$ does not depend on $\omega \in
\Omega$ for all $e\in F$. Moreover we show how to determine a finite
region $F'\supset F$ such that $(X^{(\omega,N)}_{0,0}(u))_e$ does
not depend on the values of $u_{n,k}$ for $e_k\notin F'$.

Since the $\calR_{p,q}$-measures are invariant in this way we prove
that they are all equal and we give a perfect simulation of them on
any finite subset of $E$. In particular we prove
$\calR_{p,q}=\calW_{p,q}=\{\phi_{p,q}^0\}$.






\section{Simplicial graph}\label{sec:grafisimpliciali}

In this section we define the notion of simplicial graph and we
prove some geometric properties of these graphs. It is possible that
these results are already known, maybe with different notations,
however we could not find a reference.

The prototypical graph we have in mind is the graph $\mL^d$ whose
vertices's are the elements of $\mZ^d$ and whose edges are the
segments of length one joining them. More in general a simplicial
graph will be the graph obtained by considering vertices's and
segment of a polyhedral tessellation of $\mR^d$.

Before giving the details we explain roughly the problem we want to
consider in the next sections. Let $G=(V,E)$ be a graph and color
each edge of $G$ white or back in a random way. Now consider a
finite subset $F$ of $E$. We want to determine a region $\bar H$
containing $F$ such that for all $e \in F$ if there is a path of
black edges joining the two end vertexes of $e$ then there is a path
contained in $\bar H$ of black edges joining the two end vertexes of
$e$ (see Theorem~\ref{teorema1}). Moreover we want $\bar H$ to be
small as possible so that if the probability of an edge to be black
is high then if $F$ is finite then $\bar H$ is also almost surely
finite (see Theorem~\ref{teo:pqbassa}). We construct first a set $H$
by adding inductively to $F$ white edges until its ``boundary'' is
entirely composed by black edges and then we set $\bar H$ to be
equal to the union of $H$ with its boundary. The idea is that, in
this way, if two points are in the boundary of $\bar H$, then they
can be connected by a path contained in the boundary and in
particular of black edges (for the precise statement see Proposition
\ref{prp:contorno2}). In this way when a path of black edges has its
end vertexes in $F$ can be replaced by a path of black edges with
the same end vertexes and contained in $\bar H$ by replacing the
pieces outside $\bar H$ with paths along the boundary. A first try
could be to construct $H$ by adding edges to $F$ until there are
white edges connected by a vertex to the set. In this way $H$ would
be the union of $F$ with the connected components of the subgraph of
white edges having non trivial intersection with $F$. However it is
immediate to see that this $H$ has not the required properties. For
a simplicial graph we can construct $H$ replacing the condition ``to
have a vertex in common'' with a different condition. We construct
$H$ by adding to $F$ white edges until there are white edges in the
``boundary'' of this set. The correct notion of ``boundary'' is
defined in Section \ref{ssec:poliedrali}, now we explain it in the
case of the simplicial graph $\mL^d$. We say that an edge $e'$ is in
the ``boundary'' of an edge $e$ if they are different and if there
exists a $d$-dimensional hypercube of side $1$ with vertexes in
$\mZ^d$ containing both $e$ and $e'$. In this section we prove that
this notion of boundary has the required geometrical properties (see
Proposition \ref{prp:contorno2}).

\subsection{Definitions}

Let $A$ be a closed convex subset of $\mR^\ell$ which is the
intersection of a finite number of closed half-spaces. Such a set
will be called a \emph{convex cell} and will be the starting point
of our constructions. For such a set we can identify the subset of
vertexes, edges and $i$-dimensional faces and we denote by $A_i$ the
set of $i$-dimensional faces of $A$.

In some constructions will be useful to have a more general notion
of cell. If $A$ is a convex cell of dimension $m$ and $\grf:A\lra
\mR^d$ is a piecewise affine continuous injective map we call the
image $\grs$ of $A$ a \emph{cell} of $\mR^d$ of dimension $m$.  We
define also the collections $\grs_i=\{\grf(B)\st B\in A_i\}$ and the
datum of $\grs_0,\dots,\grs_m$ will be called a \emph{polyhedron}.

The assumption piecewise affine on $\grf$ could be highly relaxed,
however this assumption makes some of the arguments below more
elementary and does not change the generality of our applications.

A \emph{polytope} $\calP$ in $\mR^d$ is a collection of cells in
$\mR^d$ which intersect properly. More precisely $\calP$ is the
datum of sets $P_0,P_1,\cdots,P_m$ such that

\begin{enumerate}[\indent $i$)]
\item the elements of $P_i$ are $i$-dimensional cells of $\mR^d$;
\item $P_i$ is locally finite: this means that for all bounded regions $R$ of $\mR^d$
$\grs \cap R = \vuoto$ for  all $\grs \in P_i$ but a finite number;
\item if we denote by $P_j(\grs) = \{\tau \in P_j\st \tau \subset \grs
  \}$, for all $\grs \in P_i$ and all $j\leq i$, then the collection
  $$ \calP(\grs)=\big\{P_0(\grs),\dots,P_{i}(\grs)\big\} $$
  is a polyhedron and the set $P_i(\grs)$ will be called the set of $i$-faces of $\grs$;
\item for all $\grs\in P_i$ and $\tau \in P_j$ the intersection
  $\grs\cap\tau$ is either empty or a union of faces of $\grs$ and $\tau$.
\end{enumerate}

The elements of $P_i$ will be called the $i$-cells of $\calP$ and in
particular we will call $P_0$ (resp.\ $P_1$) the set of vertexes
(resp. of edges) of $\calP$. The union of all the cells of $\calP$
will be called the $\emph{support}$ of $\calP$ and will be denoted
by $\supp{\calP}$.

The graph $G(\calP)=(V,E)$ associated to $\calP$ is defined as
follows: $V=P_0$, $E=P_1$ and the end vertexes of an edge $e\in E$
is the pair of vertexes contained in $e$. Notice that if $x$ and $y$
are two vertexes of $\calP$ then they are in the same connected
component of $\supp{\calP}$ if and only if they are in the same
connected component of $G(\calP)$.

The simplest possible convex cell are the simplexes. The $\ell$
dimensional standard simplex is the set $S=\{(x_0,\dots,x_\ell)\in
\mR^{\ell+1}\st x_i\geq 0$ for all $i$ and $\sum_i x_i=1\}$. A
$\ell$--dimensional \emph{convex simplex} (resp. an
$\ell$--dimensional \emph{simplex}) is the image of $S$ under an
affine (resp. piecewise affine and continuous) injective map. We
notice that every cell can be obtained as the support of a polytope
whose cells are simplexes.

A refinement of a polytope $\calP$ is a polytope $\calP'$ such that
$\supp{\calP}=\supp{\calP'}$ and each cell of $\calP'$ is contained
in a cell of $\calP$.

\subsection{Internal and external part of codimension one smooth polytopes}

We say that a polytope $\calC$ in $\mR^d$ is \emph{smooth} if its
support is smooth as a topological variety. In this case this means
that for every $x \in \supp{\calC}$ there exists a natural number
$j$, a neighborhood $W$ of $x$ in $\mR^d$ an open ball $W'$ of
$\mR^d$ and a piecewise affine continuous map $\psi: W \lra W'$
which defines an homeomorphism between $W$ and $W'$, such that
$\psi(x)=0$ and $\psi(W\cap \supp{\calC})= \{(t_1,\dots,t_n)\in W'
\st t_{1}=\dots=t_j=0\}$. If, moreover all the connected components
of $\supp{\calC}$ have dimension $n-1$ we say that it is a smooth
polytope of codimension one. Notice that if $\calC$ is smooth of
codimension one in $\mR^d$ then for all $x$ in $\supp{\calC}$ there
exists a neighborhood $W$ of $x$ in $\mR^d$ which is divided by
$\supp{\calC}$ into two open connected components. Now we give a
more global construction of these components.

Let $\calC$ be a codimension one smooth polytope in $\mR^d$ with a
finite number of vertexes and let $U$ be the complement of
$\supp{\calC}$. For all $x \in U$ we consider the set
$\sfS_x(\calC)$ of half-lines $\ell$ starting in $x$ and whose
intersection with $\supp{\calC}$ is generic. More precisely we
require that for all cells $\grs$ of $\calC$ if $\ell\cap \grs \neq
\vuoto$ then $\grs \in P_{d-1}$ and $\ell\cap \grs$ is a finite set
and moreover this intersection is contained in the set of points of
$\grs$ that are linearly smooth: for all $y\in \ell \cap \grs$ there
exists a neighborhood $U$ of $y$ and an hyperplane such that
$\grs\cap U = U\cap H$. Then the parity of the cardinality of $\ell
\cap \supp{\calC}$ does not depend on $\ell \in \sfS_x(\calC)$.
Moreover this parity is locally constant on $x\in U$.

Hence we define the \emph{internal part} of $\calC$, that we will
denote by $\Int \calC$ as the set of points $x\in U$ such that this
cardinality is odd and the \emph{external part}, that we will denote
by $\Est \calC$ as the set of point such that this cardinality is
even. Notice that $\Int \calC$ and $\Est \calC$ are two open subsets
of $\mR^d$ with boundary equal to $\supp{\calC}$. In particular for
all path joining an element of $\Int \calC$ with an element of $\Est
\calC$ the intersection of $\gamma$ with $\supp{\calC}$ is not
empty. Finally notice that $\Int \calC$ is bounded.

Notice that for all $x \in \supp{\calC}$ if $W$ is a neighborhood of
$x$ as in the beginning of this section, then the two connected
components of $W \senza \supp{\calC}$ are the intersection of $W$
with $\Int \calC$ and $\Est \calC$.  We say that a path $\gamma$
cross $\calC$ in $x$ if $\gamma(t_0)=x$ for some $t_0$ and there
exists sequences $\{t_n\}$ and $\{s_n\}$ going to $t_0$ such that
$\gamma (s_n) \in \Est \calC$ and $\gamma (t_n) \in \Int \calC$ for
all $n$.

\begin{lemma}\label{prp:dentrofuori} Let $\calC$ be a codimension
one smooth polytope in $\mR^d$ with a finite number of vertexes.
Then for all $x,y\in \supp{\calC}$ such that there exists a path in
$\overline {\Int \calC}$ joining $x$ and $y$, and a path in
$\overline {\Est \calC}$ joining $x$ and $y$, then $x$ and $y$ are
in the same connected component of $\supp{\calC}$. In particular if
$x,y$ are vertexes then they are in the same connected component of
$G(\calC)$.
\end{lemma}

\begin{proof}
Let $\gra$ (resp.\ $\grb$) be a path joining $x$ and $y$ in
$\overline {\Int \calC}$ (resp.\ $\overline {\Est \calC}$). Consider
the connected component $D$ of $\supp{\calC}$ containing $x$ and let
$\calD$ be the polytope whose cells are the cells of $\calC$
contained in $D$. $\calD$ is a connected codimension one smooth
polytope with a finite number of vertexes. Notice that, by what
noticed above, in a small neighborhood $W$ of $D$ we have that
$W\senza D$ has two connected components $W_1$ and $W_2$ and we have
$W_1=W\cap \Int \calC$ and $W_2=W\cap \Est \calC$. Similarly $W_1$
and $W_2$ must be the intersection with $\Int \calD$ and $\Est
\calD$, and both the possibilities
$$
\begin{cases}
W_1=W\cap \Int \calC \\
W_2=W\cap \Est \calC
\end{cases}
\qquad \mand
\begin{cases}
W_1=W\cap \Est \calC \\
W_2=W\cap \Int \calC
\end{cases}
$$
 can occur. Since $\gra$ and $\grb$ do never cross $\calC$ they
also never cross $\calD$ and we have that $\gra$ is contained in
$\overline \Int \calD$ and $\grb$ is contained in $\overline \Est
\calD$ or the opposite.

Hence the final point $y$ of $\gra$ and $\grb$ belongs to $\overline
{\Int \calD}\cap \overline {\Est \calD}$ hence it is in $D$.
\end{proof}

\subsection{Polyhedral tessellation and simplicial graph}\label{ssec:poliedrali}

We say that a polytope $\calP$ is a \emph{polyhedral tessellation}
of $\mR^d$ if $\supp{\calP}=\mR^d$. In this case we say that
$G=G(\calP)$ is a \emph{simplicial graph}.

If $A\subset P_0$ we define the \emph{boundary} $\Delta_\calP(A)$ of
$A$ as the set of vertexes $v \in P_0\senza A$ for which there
exists, $w\in A$ and a cell $\grs$ of $\calP$, such that $v,w\in
\grs$ and we define $G(A)$ as the graph whose set of vertexes is
equal to $A$ and whose edges are given by the element in $P_1$
joining two elements of $A$. Notice also that if $x\in A$ and $y\in
P_0\senza A$ any path in $G(\calP)$ joining $x$ and $y$ intersects
$\Delta_\calP(A)$.

\begin{prop}\label{prp:contorno1}
Let $\calP$ be a polyhedral tessellation of $\mR^d$ and let $V$ be
its set of vertexes. Let $A$ be a finite subset of $V$ and set
$B=V\senza A$.  Let $x,y \in A$ adjacent respectively to $x', y'\in
B$. Assume now that $x$, $y$ are connected in $G(A)$ and that $x'$,
$y'$ are connected in $G(B)$.  Then $x',y'$ are connected in
$G(\Delta_\calP(A))$.
\end{prop}

\begin{proof} The proof follows exactly the same lines in the case of a
tessellation using convex cells and in the case of a general
tessellation. However in the first case all construction are more
intuitive and direct.  For this reason we give first the proof in
the case of a tessellation made of convex cells and then we briefly
explain how to change the proof in the general case.

\medskip

In the first step of the proof we assume also that all cells are
convex simplexes.  We construct a smooth polytope $\calC$ of
codimension one that separate $A$ and $B$ in the following way.

We define the set of vertexes $C_0$ of $\calC$ as the set of middle
points of edges joining an element of $A$ and an element of $B$. For
all $i$ dimensional simplexes $\grs$ of $\calC$ which contain an
element of $C_0$ the convex envelope of $\grs \cap C_0$ is a $i-1$
cell. Let $C_{i-1}$ be the collection of these cells and let $\calC$
be the polytope whose $i$ dimensional cells are given by $C_i$ for
$i=0,\dots,n-1$. By construction $\calC$ is a smooth polytope of
codimension one with a finite number of vertexes. Notice also that a
path in $G(A)$ or in $G(B)$ will never cross $\supp{\calC}$.

Let now $u,v\in C_0$ be the middle points of the edges joining $x$,
$x'$ and $y$ and $y'$ respectively. By Proposition
\ref{prp:dentrofuori} $u$ and $v$ are in the same connected
component of $G(\calC)$. Hence there exists a sequence of vertexes
$w_0=u, w_1, \dots, w_m=v$ in $C_0$ determining the path connecting
$u$ and $v$ in $G(\calC)$. Furthermore let $t_i\in A$ and $t'_i\in
B$ be such that $w_i$ is the middle point of the edge joining $t_i$
and $t'_i$.  Then $t'_0=x', t'_1,\dots, t'_m=y'$ determine a path in
$\Delta_\calP(A)$ joining $x'$ and $y'$.

\medskip

Let now $\calP$ be any tessellation with convex cells. We construct
a sequence $\calP^{(i)}$ of refinements of $\calP$ and of finite
subsets $A^{(i)}$ of the vertexes of $\calP^{(i)}$.

\begin{itemize}
\item $\calP^{(1)}$ is $\calP$ and  $A^{(1)}=A$.
\item $\calP^{(2)}$ is the tessellation obtained by adding a vertex
  $v_\grs$ in the barycentre of all $2$-dimensional faces $\grs\in P_2$ which
  are not a simplex and adding the edges joining $v_\grs$ with the
  vertexes of $\grs$. Finally we set $A^{(2)} = A \cup \{v_\grs
  \st \grs \cap A \neq \vuoto\}$.
\item more generally given $\calP^{(i-1)}$, $\calP^{(i)}$ will be the
  tessellation obtained adding a vertex $v_\grs$ in the barycentre of
  every $i$-dimensional cell $\grs$ of $\calC^{(i-1)}$ which is not a
  simplex and adding all the $j$--dimensional cells obtained by joining
  this vertex with the $j-1$ cells contained in $\grs$. Finally we set
  $A^{(i)} = A^{(i-1)} \cup \{v_\grs \st \grs \cap A^{(i-1)} \neq
  \vuoto\}$.
\end{itemize}

Set $\calP'=\calP^{(d)}$ and $A'=A^{(d)}$, and $B'$ is the
complement of $A'$ in the set of vertexes of $\calP'$. Notice that
all cells of $\calC'$ are convex simplexes and that
$\Delta_\calP(A)=\Delta_{\calP'}(A')$. Hence we can apply to this
situation what we have already proved.

\medskip

In the case of a tessellation whose cells are not convex some of the
constructions we have described have to be modified. Indeed it does
not make sense to consider the middle point of an edge or the convex
envelope of the middle points in the first part of the proof or the
barycentre in the second part of the proof. As an example we show
which changes are necessary in the construction of the sequence
$\calP^{(i)}$.

Suppose we have already constructed $\calP^{(i-1)}$ and that all
$j$-dimensional cells of $\calP^{(i-1)}$ are simplexes for $j\leq
i-1$. Now consider a $i$-dimensional cell $\grs$ of $\calP^{(i-1)}$
which is not a simplex. Let $\grf:A\lra \grs$ be the piecewise
affine map which parametrise $\grs$ where $A$ is a convex cells.
Notice that by induction all faces of $A$ are convex simplexes. Now
let $v$ be the barycentre of $A$ and that we can divide $A$ into
simplexes $A_\tau$ joining $v$ with the faces of $A$. We obtain the
new tessellation by considering the restriction of $\grf$ to the
simplexes $A_\tau$.
\end{proof}

The result we will use in the next section is a variation of
Proposition \ref{prp:contorno1} where the set of vertexes $A$ is
replaced by a set of edges.

If $H\subset P_1$ we define the \emph{boundary} $\Delta_\calP(H)$ of
$H$ as the set of edges $e \in P_1\senza H$ for which there exists,
$d\in H $ and $\grs \in P_n$ such that $e,d\subset \grs$. Let also
$V_H=\{v\in P_0\st v\in e$ for some $e\in H\}$ and let $G(H)$ be the
subgraph $(V_H,H)$ of $G(\calP)$.

\begin{prop}\label{prp:contorno2} Let $\calP$ be a polyhedral tessellation of $\mR^d$
and $H \subset P_1$ be a finite set. Let $x,y \in V_H$ be connected
in $G(H)$ and in $G(P_1\senza H)$. Then $x$ and $y$ are connected in
$G(\Delta_\calP(H))$.
\end{prop}

\begin{proof}
Let $e_1,\dots,e_m$ be a path in $H$ which join $x$ and $y$ and let
$f_1,\dots,f_n$ be a path in $P_1\senza H$ which joins $x$ and $y$.
Let $x=u_0,\dots,u_m=y$ and $x=v_0,\dots,v_n=y$ be the sequences of
adjacent vertexes determined respectively by the path in $H$ and in
$P_1\senza H$. Dividing the path into smaller pieces we can assume
that $u_i\neq v_j$ for $1\leq i\leq m-1$ and $1\leq j \leq n-1$.

We construct a refinement $\calQ$ of $\calP$ replacing each vertex
$w$ with a small convex cell so that we will be able to apply
Proposition \ref{prp:contorno1} to this situation.

More precisely for each vertex $w$ we consider a ``small'' ball
$B_w$ with centre $w$ which is ``far'' from every cell of $\calP$
not containing $w$. We set also $D_w$ to be the intersection of the
boundary of $B_w$ with the edges in $P_1$ containing $w$ and $A_w$
to be the intersection of the boundary of $B_w$ with the edges in
$H$ containing $w$.  We set
$$
Q_0  = P_0 \cup \bigcup_{w\in P_0} D_v \; \text{ and }\; A  =
\{u_i:i=0,\dots,m\} \cup \bigcup_{w\in P_0}A_w.
$$

We now describe the higher dimensional cells of $\calQ$. Let $F_w$
be convex envelope of the set $D_w$ and let $M_n$ be its boundary
and $N_w$ its open part. If $B_w$ is small enough for every
$i$--dimensional cell $\grs$ of $\calP$ containing $w$ the
intersections $\grs_w=\grs\cap F_w$ and $\grs_w^\partial=\grs\cap
M_w$ are respectively an $i$-dimensional cell and an
$(i-1)$--dimensional cell.  Finally for all $i$--dimensional cells
$\grs$ of $\calP$ the set $\tilde \grs = \grs \senza \bigcup_w N_w$
is also an $i$--dimensional cell. The tessellation $\calQ$ is the
collection of the cells $\tilde \grs$, $\grs_w$ and
$\grs_w^\partial$ for $\grs$ a cell of $\calP$ and $w\in P_0$.

Notice that if $\tilde \grs$ is an edge in $G(\Delta_\calQ(A))$ then
$\grs$ is an edge in $G(\Delta_\calP(H))$.

We now apply Proposition \ref{prp:contorno1}. The sequence of edges
$$
(e_1)_{u_0}, \tilde e_1, (e_1)_{u_1}, (e_2)_{u_1},\tilde e_2,\dots,
\tilde e_m,(e_m)_{u_m}
$$ is a path in $G(A)$ joining $x$ and $y$. Let
$x'=(f_1)^\partial_{x}$ and $y'=(f_n)^\partial_{y}$ then they are
two vertexes adjacent respectively to $x$ and $y$ and the sequence
of edges
$$
\tilde f_1, (f_1)_{v_1}, (f_2)_{v_1},\tilde f_2,\dots,
(f_n)_{v_{n-1}},\tilde f_n
$$
is a path in $G(Q_0\senza A)$ joining $x'$ and $y'$. Hence there
exists a path
$$
\tau_1,\dots,\tau_{r_1},\tilde
\grs_1,\tau_{r_1+1},\dots,\tau_{r_2},\tilde \grs_2,\dots, \tau_{r_d}
$$
in $G(\Delta_\calQ(A))$ joining $x'$ and $y'$ where we assume that
$\tau_i$ is of the form $\eta_w$ or $\eta_w^\partial$ for all $i$.
Then $\grs_1,\dots,\grs_{d-1}$ is a path in $G(\Delta_\calP(H))$
joining $x$ and $y$.
\end{proof}






\section{Coupling of the random cluster measure at low or high temperatures}\label{sec:bassa}

In this section we fix a polyhedral tessellation $\calP$ in $\mR^d$
and we denote by $G=(V,E)$ the underlying simplicial graph. We
introduce a new graph $G^{*} =( V^{*} ,E^{*} )$ where $V^{*} = \mZ
\times V$ and $E^{*} = \mZ \times E\cup \mZ\times V$ where if $x,y$
are the end vertexes of $e\in E$ then $(n,x),(n,y)$ are the end
vertexes of $(n,e)\in E^*$ and if $v\in V$ and $n\in \mZ$ then we
denote with $e_{n,v}$ the corresponding element in $E^*$ and its end
vertexes are $(n,v),(n+1,v) \in V^*$. Notice that $G^*$ is the
simplicial graph of a polyhedral tessellation $\calP^*$ of
$\mR^{d+1}=\mR\times \mR^d$ whose $i$--dimensional cells are the
collection of the cells of the form $\{n\}\times \grs$ where $\grs$
is an $i$--dimensional cell of $\calP$ and $n\in \mZ$ and of the
cells of the form $[n,n+1]\times \tau$ and $\tau$ is an
$(i-1)$--dimensional cell of $\calP$ and $n\in \mZ$.  We consider
$G$ as the subgraph $G\times \{ -1\}$ of $G^*$.

Fix $u \in \calU$. We define the following coloring of the edges
$E^*$ of $G^*$:
\begin{align*}
W & = W(u) = \{(n,e_k)\in E^* \st (n,k)\in \calA \mand p_{e_k}\leq u_{n,k} \leq 1 \},\\
M & = M(u) = \{(n,e_k)\in E^* \st (n,k)\in \calA \mand \hat p_{e_k} \leq u_{n,k} < p_{e_k}\},\\
B & = B(u) = E^*\senza (M(u)\cup W(u)).
\end{align*}
We define also $B_n=\{e\in E\st (n,e)\in B\}$. We say that the
elements of $B$ are black, the elements of $M$ are gray and the
elements of $W$ are white.

Given a subset $F$ of $M(u) \cup W(u)$ we define the \emph{ cluster
of white or  gray edges } $C_F^w=C_F^w(u)$ as the minimum subset $H$
of $E^*$ containing $F$ such that if $e \in E^*$ and $e \in
\Delta_{\calP^*}(H)$ then $e\in B(u)$.  More in general if $F\subset
E$ is not necessarily a subset of $M \cup W$ we define $C_F^w=F\cup
C_{F\senza B}^w$. Equivalently if $F\subset M$ we can construct
$C_F^w$ inductively by adding the white or gray edges ``near'' to
$F$ as follows. Let $D_0=F$ and $D_{i+1} = D_i \cup
(\Delta_{\calP^*}(D_i)\senza B)$ then $C_F^w = \bigcup_i D_i$.


Notice that as the numbers $p_e$ grow the probability that $C^w_F$
is finite increase.  As we will see in the last section in the Ising
model the $p_e$'s are related to a parameter called temperature and
when this parameter is small the $p_e$ are closer to $1$. For this
reason we refer to the case in which $C^w_F$ is finite as the
situation at low temperature.

Assume now that $C_F^w$ is finite. In this case we define $N^w(u,F)$
as the biggest negative integer $N$ such that $C_F^w \subset E
\times [N+1,-1]$. We define
\begin{equation}\label{eq:H}
 H_n^w = H_n^w(u,F)= \big\{e \in E \st
(n,e) \in C_F^w \text{ or } (n-1,e)\in C_F^w \big\}
\end{equation}
and we set $\bar H_n^w(u,F)= H_n^w \cup \Delta_\calP (H_n^w)$ and
finally $\bar H^w(u,F) = \bigcup_n \{n\}\times \bar H^w_n $.

\begin{theorem}\label{teorema1}
Fix $u \in \calU$ such that $C_F^w(u)$ is finite and an integer
$N\leq N^w(u,F)$. Let $u'\in \calU_N$ such that $u'_{n,k} = u_{n,k}$
for all $(n,e_k) \in \bar H^w(u,F)$. Then for all $\omega,\omega'
\in \Omega$ we have
$$ \big(X_{0,0}^{(\omega,N)}(u)\big)_e = \big( X_{0,0}^{(\omega',N)}(u')\big)_e $$
for all $e\in \bar H_{-1}^w(u,F)$.
\end{theorem}

\begin{proof}
Let $\eta_{n,k}=X_{n,k}^{(\omega,N)}(u)$ and
$\eta'_{n,k}=X_{n,k}^{(\omega',N)}(u')$.  Let also $C = C_ F^w$ and
$C_n = \{e\in E\st (n,e) \in C\}$, $H_n =C_n\cup C_{n-1}$ and $\bar
H_n =H_n\cup \Delta_\calP(H_n)$. We notice first that
$\Delta_\calP(H_n)\subset B_n \cap B_{n-1}$. Indeed let $e\in
\Delta_\calP(H_n)$, then there exists $e' \in H_n$ and a cell $\grs$
in $\calP$ such that $e$ and $e'$ are contained in $\grs$. If
$e\notin B_n$ then by the definition of $C$ and of the cells of
$\calP^*$ we have that $(n,e)\in C$, hence $e\in C_n$ which is in
contradiction with $e\in \Delta_\calP(H_n)$.  Similarly we get an
absurd if $e \notin B_{n-1}$. By the definition of a
$FK^N_{p,q}$--process this implies that
\begin{equation}\label{DeltaH}
\Delta_\calP(H_m)\subset E(\eta_{m,h})\cap E(\eta'_{m,h})
\end{equation}
for all $m > N$ and $h\geq 0$.

We will prove that
\begin{equation}\label{eqX}
   \big(\eta_{m,h}\big)_e = \big(\eta'_{m,h}\big)_e
\end{equation}
for all $(m,h) \in \tilde \calA_N$ and for all $e\in \bar H_m$. We
prove this by induction starting with $m=N$ and $h=0$. For $m\leq
N^w(u,F)$ and for all $h$ the equality \eqref{eqX} is trivially
satisfied since $C_{m}=\vuoto$.

Now we prove that if equation \eqref{eqX} holds for $m=N,\dots,n-1$
and for all $h$ and for $m=n$ and $h=0,1,\dots, k$ then it holds
also $m=n$ and $h=k+1$.  Let $e\in \bar H_n$. If $e\neq e_{k+1}$
then by induction and definition of $FK^N_{p,q}$--process we get
$$
\big(\eta_{n,k+1}\big)_e  =  \big(\eta_{n,k}\big)_e  =
\big(\eta'_{n,k}\big)_e  =  \big(\eta'_{n,k+1}\big)_e
$$
proving the claim. If $e=e_{k+1}$ and $e\in \bar H_n$ we compute
$\big(\eta_{n,k+1}\big)_{e_{k+1}}$. If $u_{n,k+1}\geq p_{e_{k+1}}$
then we have $\big(\eta_{n,k+1}\big)_{e_{k+1}} =0$ and similarly for
$\eta'$. If $u_{n,k+1}< \hat p_{e_{k+1}}$ then we have
$\big(\eta_{n,k+1}\big)_{e_{k+1}} =1$ and similarly for $\eta'$. If
$\hat p_{e_{k+1}} < u_{n,k+1}\leq p_{e_{k+1}}$ we need to prove that
\begin{equation}\label{KiffK}
\eta_{n,k}  \in K_{e_{k+1}} \quad \text{ iff } \quad \eta'_{n,k} \in
K_{e_{k+1}}.
\end{equation}
Let $x$ and $y$ be the end vertexes of $e_{k+1}$ and assume that
there is a path $\gamma:\gre_1,\dots,\gre_m$ in $E(\eta_{n,k})\senza
\{e_{k+1}\}$ joining $x$ and $y$.

If the path is contained in $\bar H_n=C_n\cup C_{n-1} \cup
\Delta_\calP (H_n)$ then we prove that the same path is contained in
$E(\eta'_{n,k} )\senza \{e_{k+1}\}$. Let $e_r$ be an edge of the
path. By induction $(\eta_{n,k}) _{e_r} =(\eta'_{n,k})_{e_r}$ hence
$e_r\in E(\eta'_{n,k} )$.

If $\gamma$ is not contained in $\bar H_n$ we show there is another
path joining $x$ and $y$ contained in $\bar H_n$. By \eqref{DeltaH}
and the fact that $\hat p_{e_{k+1}} < u_{n,k+1}\leq p_{e_{k+1}}$ we
have $e_{k+1}\in H_n$. Let $D$ be the connected component of $H_n$
containing $e_{k+1}$. Let $\gre_1,\dots,\gre_{i-1} \in D$, $\gre_i,
\dots, \gre_j \notin D$ and $\gre_{j+1}\in D$. Let $x'$ be the
vertex common to $\gre_{i-1}$ and $\gre_i$ and $y'$ the vertex
common to $\gre_j$ and $\gre_{j+1}$.  We can apply Proposition
\ref{prp:contorno2} and we construct a path $\grb$ in
$\Delta_\calP(D)\subset \Delta_\calP(H_n)$ joining $x'$ and $y'$.
Since $\Delta_\calP(H_n)\subset E(\eta_{n,k})\cap E(\eta'_{n,k})$ we
can replace $\grg$ with the path
$\grg':\gre_1,\dots,\gre_{i-1},\grb,\gre_{j+1}, \dots, \gre_m$.
Repeating this process we see that we can substitute the path
$\gamma$ with a path entirely contained in $D\cup
\Delta_\calP(D)\subset \bar H_n$ as claimed. Hence we are reduced to
the previous case.

\medskip

Finally we prove that if \eqref{eqX} holds for a fixed $m$ and all
$h\geq 1$ then it holds also for $m+1,0$. Let $e \in C_{m+1}$. If $e
\in \bar H_{m}$ this follows by definition of the
$FK^N_{p,q}$-process. If $e=e_r \in \bar H_{m+1}\senza C_m$ then $e
\in B_{m}$ otherwise $e$ would be an element of $C_{m}$). Then
$$(\eta_{m+1,0})_{e_r}=(\eta_{m,r})_{e_r}=1=(\eta'_{m,r})_{e_r}=(\eta_{m+1,0})_{e_r}$$
proving the claim.
\end{proof}

\subsection{Coupling at high temperatures}\label{ssec:alta}

We give now a similar result corresponding, in the Ising model, to
high temperatures.

Let $G=(V,E)$ be a countable graph (in this case we do not assume
simplicial). Define $\bar G$ as the graph with set of vertexes $\bar
V= \mZ_{< 0}\times V$ and edges $\bar V= \mZ_{< 0} \times E$ where
if $e \in E$ has end vertexes $x,y$ then the edge $(n,e)$ has end
vertexes $(n,x)$ and $(n,y)$.  We consider $G$ as the subgraph
$G\times \{ -1\}$ of $G^*$. For all $u \in \calU$ define $M(u),W(u)$
as in the previous section.





For a subset $H$ of $E$ we denote by $\Gamma(H)$ the set of edges
which are not in $H$ and which have a vertex in common with an edge
in $H$.

Fix $u\in \calU$ and a subset $F$ of $E^*$. If $F\subset M(u)$ we
define \emph{ the cluster of black or gray edges } $C^b_F=C^b_F(u)$
as the smallest set $C$ of $\bar E$ containing $F$ and such that for
all $(n,e)\in \bar E \senza C$ if either $(n-1,e)$ or $(n,e)$ or
$(n+1,e)$ have a vertex in common with an edge in $C$ then $(n,e)\in
W(u)$. If $F$ is not necessarily contained in $M(u)$ we define
$C_F^b(u)=C^b_{F\cap M(u)}(u)\cup F$. For all $n<0$ we define
$$ H_n^b = H_n^b(u,F)= \big\{e \in E \st (n,e) \in C_F^b(u) \text{ or
} (n-1,e)\in C_F^b(u) \big\}$$ and we set $\bar H_n^b(u,F)= H_n^b
\cup \Gamma (H_n^b)$ and $\bar H^b(u,F) = \bigcup_n \{n\}\times \bar
H^b_n $.

Finally if $C_F^b(u)$ is finite we define $N^b(u,F)$ as the biggest
negative integer $N$ such that $C_F^b \subset E \times [N+1,-1]$.

\begin{theorem}\label{teorema2}
Fix $u \in \calU$ such that $C_F^b(u)$ is finite and an integer
$N\leq N^b(u,F)$. If $\omega, \omega' \in \Omega$ and $u'\in
\calU_N$ is such that $u'_{n,k} = u_{n,k}$ for all $(e_k,n) \in \bar
H^b(u,F)$ then
$$ \big(X_{0,0}^{(\omega,N)}(u)\big)_e = \big( X_{0,0}^{(\omega',N)}(u')\big)_e $$
for all $e \in \bar H^b_{-1}(u,F)$.
\end{theorem}

\begin{proof}
The proof follows exactly the same strategy of the proof of
Theorem~\ref{teorema1}. However in this case it is simpler since we
do not have to use the result of Section~\ref{sec:grafisimpliciali}.
We give here only the main lines.  Indeed an argument analogous to
proof of equation~\eqref{DeltaH} gives
\begin{equation}\label{GammaH}
\Gamma(H^b(u,F))\subset W(u).
\end{equation}
Then we prove the equality~\eqref{eqX} by induction as in the proof
of Theorem~\ref{teorema1}. Also in this case we are reduced easily
to prove the equivalence~\eqref{KiffK}. This equivalence is easier
in this case, since, by~\eqref{GammaH}, we have that
$E(\eta_{n,h})\subset H^b_n$ for all $n>N$ and similarly for $\eta'$
so we can assume that the path joining the extremal point of
$e_{k+1}$ is contained in $H^b_n$ without using any further result,
while in the proof of Theorem~\ref{teorema1} we need Proposition
\ref{prp:contorno2}.  The remaining argument are completely similar
to the proof of Theorem~\ref{teorema1}.
\end{proof}






\section{Assumptions for the finiteness of clusters}\label{sec:clusterfiniti}

In this section, given a graph $G$, we present some conditions on it
and on the parameters $p$ such that the cluster $C^w_F$ of
Theorem~\ref{teorema1} is almost surely finite or such that the
cluster $C^b_F$ of Theorem~\ref{teorema2} is almost surely finite.

We start by recalling a general Lemma.  Let $\mG=(\mV,\mE)$ be a
graph and let $\pi=(\pi_v)_{v\in \mV}$ be an element of $[0,1]^\mV$.
Consider the product measure on the space $\Omega_\mG=\{0,1\}^\mV$
such that $P(\omega_v=0) = \pi_v$. For each $\omega \in \Omega_\mG$
let $\mG[\omega]$ be the subgraph of $G$ with set of vertexes
$\mV[\omega]=\{v\in \mV\st \omega_v = 0\}$ and with edges the set
$\mE[\omega]$ of the elements of $\mE$ joining two elements in
$\mV[\omega]$. Moreover for each $\omega \in \Omega_\mG$ and for any
$v\in \mV$ set $\mG_v[\omega]$ the connected component of
$\mG[\omega]$ containing $v$ (possibly empty if $v\not\in
\mV[\omega]$) and if $n$ is a natural number let $\mV_{v,n}[\omega]$
be the set of vertexes in $\mG_v[\omega]$ whose graph-distance from
$v$ in the graph $\mG_v[\omega]$ is equal to $n$.

\begin{lemma}\label{lem:casosemplice}

Let $\mG$ and $\pi$ be as above. For each vertex $v\in \mV$ let
$A_v$ be the set of vertexes adjacent to $v$ and set $g_v =
\sum_{v\in A_v } \pi_v$. If $g=\sup \{g_v\st v\in \mV \}< 1$ then
$P(\{\omega\in \Omega_\mG \st \mV_{v,n}[\omega] \neq \vuoto\} )\leq
g^n$.
\end{lemma}

\begin{proof}
Define the random variables $Z_{v,n}[\omega] = \card
\big(\mV_{v,n}[\omega]\big)$. The conditional mean value
$E(Z_{v,n+1}|Z_{v,n})$ verifies $ E(Z_{v,n+1}|Z_{v,n})\leq g\,
Z_{v,n}$, in particular the sequence $\{ Z_{v,n} \}_n$ is a
supermartingale. Hence the mean value $E(Z_{v,n})$ is less or equal
to $g^n$. By Markov inequality the claim follows.
\end{proof}



\subsection{Assumptions for the finiteness of clusters at low temperatures}

Now we give conditions on $G$ and $p$ such that the cluster
$C^w_F(u)$ defined in Section~\ref{sec:bassa} is finite for almost
all $u\in \calU$. We fix a polyhedral tessellation $\calP$ of
$\mR^d$. Let $G =(V,E)$ be the associated simplicial graph and let
$\calP^*$ and $ G^* =(V^*, E^*)$ be defined as in
Section~\ref{sec:bassa}.

For the proof of our next Theorem we introduce a new graph
$\mG=(\mV, \mE)$ defined as follows: $\mV= E^*$ and two elements
$e,e'\in \mV$ are joined by an edge in $\mE$ if and only if there
exist a cell in $\calP^*$ which contains $e $ and $e'$.

If $H\subset E^*$ and $u\in \calU$ we define also the subgraph
$\mG(H,u) = (\mV(H,u),\mE(H,u))$ of $\mG$ whose set of vertexes are
equal to $\mV(H,u) = H \senza B(u)$ and whose edges $\mE(H,u)$ are
all the edges of $\mE$ joining two vertexes in $\mV(H,u)$.

For $e \in E$ define
\begin{equation}\label{defg}
 \hat g _e = 2(1-\hat{p}_e) +3 \sum_{e' \in \Delta_\calP(\{e\})} (1-
 \hat{p}_{e'}),
\end{equation}
where $\hat{p}_e$ is defined in \eqref{eq:phat}.

\begin{theorem}\label{teo:pqbassa} Let $G=(V,E)$ be a simplicial graph in
$\mR^d$. Assume that $\hat p_e>0$ for all $e\in E $ and that
$$
\lim_{\Lambda \uparrow E} \sup_{ e \not \in \Lambda} \hat{g}_e<1.
$$
Then for all $e\in E$ the cluster $C^w_e(u)$ is finite for almost
all $u\in\calU$.
\end{theorem}

\begin{proof}
First we do a preliminary remark: the event $I=\cup_{e \in E } \{u
\in \calU : \card\big(C^w_e(u)\big) = \infty \}$, is in the tail
$\sigma$-algebra. Therefore, by Kolmogorov $0$-$1$ law, this event
has probability zero or one, in particular to prove our claim it is
enough to prove that $P(I)<1$.

Define $F_n = \{ e \in E: \hat g_e > 1- \frac{1}{n} \}$.  By
assumption there exists $ n_0 $ such that $ \card(F_{n_0}) < \infty$
and set $F=F_{n_0}$ and $g=1-\tfrac{1}{n_0}$. Set also $\tilde F=F
\cup \Delta_\calP(F)$ and, for $\ell\geq 1$, define $ \hat F = \mZ
\times F \subset \mZ\times E\subset E^*$. Fix an edge $\hat e \in E$
and, for $\ell\geq 1$  set
$$
\Lambda_\ell=[-\ell,-1]\times B_\mG(\hat e,\ell) \subset \mZ\times E
$$
and set also $S_\ell = \card(\Lambda_{\ell})$.

Choose $\ell_0$ such that $\frac{3\,\card(\tilde
  F)}{1-g}g^{\ell_0}<\frac{1}{2}$ and, for $\ell\geq \ell_0+1$, define the events
\begin{align*}
\calW_\ell &= \{u\in \calU \st u_{n, k} < \hat p_{e_k} \hbox{ for
  any } (n, e_k) \in [-\ell - S_{\ell} ,-\ell-S_\ell+\ell_0]\times F \} \\
\calX_\ell &= \{ u \in \calU :
\card\big(C^w_{\Lambda_{\ell}}(u)\big)=\infty \}.
\end{align*}

We notice that the probability $P_0 = P( \calW_\ell )$ does not
depend on $\ell$ and it is a positive constant being $ \hat p_e >0$
for any $e\in E$. We notice also that $\calX_\ell\subset
\calX_{\ell+1}$ and that their union is equal to $I$. Therefore
$$
P(I) = P\left(\bigcup_\ell \calX_\ell\right) = \lim_{\ell \to
\infty} P(\calX_\ell).
$$
We also define the events
\begin{align*}
\calY_\ell =\{& u \in \calU : \text{ there exists a sequence }
e_1,\dots,e_m \in W(u)\cup M(u)\senza \hat F \\
& \text{such that } e_i \text{ is adjacent to } e_{i+1} \text{ in the graph } \mG, \\
& e_1 \in \Lambda_\ell \text{ and }
e_m \notin  \Lambda_{\ell+S_\ell} \}  \\
\tilde \calZ_{\ell,i} =\{& u \in \calU : \text{ there exists a
sequence } e_1,\dots,e_m \in W(u)\cup M(u)
\text{ such that } \\
& e_i \text{ is adjacent to } e_{i+1} \text{ in the graph } \mG,  \\
& e_1 \in \{-\ell-S_\ell+i\}\times F \text{, } e_m \notin
\Lambda_{\ell+S_\ell} \text{ and } e_i \notin \hat F \text{ for }
1<i<m\}.
\end{align*}
Finally define $\calZ_{\ell,1}=\tilde \calZ_{\ell,1}$ and
$\calZ_{\ell,i}=\tilde \calZ_{\ell,i}\senza \tilde \calZ_{\ell,i-1}$
for $i>1$. It is clear that
\begin{equation*}\label{XYZ}
 \calX_\ell \subset \calY_\ell \cup \bigcup_{i=1}^{\ell}
 \calZ_{\ell,i},
\end{equation*}
in particular $ P(\calX_\ell | \calW_\ell) \leq P( \calY_\ell |
\calW_\ell ) + \sum_{i=1}^\ell P( \calZ_{\ell,i} | \calW_\ell)$.




Now we notice that the events $\calW_\ell$ is decreasing meaning
that if $u\in \calX_\ell$ and $u' \in \calU$ is such that
$u'_{n,h}\leq u_{n,h}$ for all $n,h$ then $u'\in\calW_\ell$.  With a
similar definition the events $\calY_\ell$ and $\calZ_{\ell,i}$ are
increasing.  Hence, by the FKG inequality we obtain $P(\calY_{\ell}
| \calW_\ell) \leq P(\calY_\ell)$ and $P(\calZ_{\ell,i} |
\calW_\ell) \leq P(\calZ_{\ell,i})$ (see~\cite{GrimmettP},
Chapter~2). Hence, noticing that $P(\calZ_{\ell,i}|\calW_\ell) = 0$
for $i < \ell_0$, we get $P(\calX_\ell | \calW_\ell) \leq
P(\calY_\ell)+\sum_{i=\ell_0}P(\calZ_{\ell,i})$.

Now we estimate $P(\calY_\ell)$ and $P(\calZ_{\ell,i})$ using
Lemma~\ref{lem:casosemplice}. We start with $\calY_\ell$.  Consider
the random graph $\mG(\hat F ^c,u)$ and set $\pi_e=1-\hat p_e$.
Notice that $\calY_\ell\subset \bigcup_{e\in \Lambda_\ell} \{u\in
\calU\st \mV(\hat F^c,u)_{e,S_\ell}\neq \vuoto\}$ hence using
Lemma~\ref{lem:casosemplice} we get
$$
P(\calY_\ell)\leq \card(\Lambda_\ell) \, g^{S_\ell} = S_\ell \,
g^{S_\ell},
$$
for $\ell $ large enough. For $\calZ_{\ell.i}$ we proceed in a
similar way. Consider again the random graph $\mG(\hat F ^c,u)$.
Notice that in the sequence $e_1,\dots,e_m$ which appears in the
definition of $\tilde \calZ_{\ell,i}$ the subsequence
$e_2\dots,e_{m-1}$ is in $\hat F^c$ and $e_2 \in \tilde
F_{\ell,i}:=\{-\ell-S_\ell+i-1,-\ell-S_\ell+i,-\ell-S_\ell+i+1\}\times
\tilde F$. Hence $\calZ_{\ell,i}\subset  \bigcup_{e\in \tilde
F_{\ell,i}} \{u\in \calU\st \mV(\hat F^c,u)_{e,i}\neq \vuoto\}$ and
using Lemma~\ref{lem:casosemplice} we get
$$
P(\calZ_{\ell,i})\leq \card(\tilde F_{\ell,i}) \, g^i = 3\,
\card(\tilde F) \, g^i.
$$
Recall that $P_0=P(\calW_\ell)$ does not depend on $\ell$, hence we
have
\begin{align*}
P(\calX_\ell)  & = P(\calX_\ell | \calW_\ell) P(\calW_\ell) +
P(\calX_\ell | \calW_\ell^c) P(\calW_\ell^c) \leq
P(\calW^c_\ell) + P(\calX_\ell | \calW_\ell) P(\calW_\ell) \\
& \leq 1 - P_0
+ P_0 \big( P( \calY_\ell ) + \sum_{i=\ell_0}^\infty P( \calZ_{\ell,i}) \big) \\
& \leq 1 - P_0 + P_0 \big(S_\ell \, g^{S_\ell}
+\frac{3\,\card(\tilde
F)}{1-g}g^{\ell_0} \big) \\
& \leq 1 - \frac{P_0}2 + S_\ell \, g^{S_\ell} P_0.
\end{align*}
Finally notice that $ \lim_{\ell \to \infty}S_\ell \,
g^{{S_\ell}}=0$. Hence
$$
\lim_{\ell \to \infty}P(\calX_\ell)  \leq 1 - \frac{P_0}{2}<1
$$
as claimed.\end{proof}

\subsection{Assumptions for the finiteness of clusters at high temperature}

Let $G=(V,E)$ be a countable graph of finite degree. We remark that
in this case we do not need to assume that $G$ is a simplicial
graph. For all $e\in E$ define
$$
g_e = 2\, p_e + 3\,\sum _{e'\in \Gamma(\{e\})} p_{e'}.
$$

\begin{theorem}\label{teo:pqalta}

Let $G$ be a countable graph of finite degree. If $p_e<1$ for all
$e\in E$ and
$$
\lim_{\Lambda \uparrow E} \sup_{ e \not \in \Lambda} g_e<1
$$ then for all $e\in E$ the set $C^b_e(u)$ is finite for almost all
$u \in \calU$.
\end{theorem}

The proof follows exactly the same lines of the proof of Theorem
\ref{teo:pqbassa}, however we do not need any result from Section
\ref{sec:grafisimpliciali}.



\section{Perfect simulation of the random cluster measure
at low or high temperature} \label{sec:simulazione}

As an application of the previous results we now explain how to
prove uniqueness of the random cluster measure and how to obtain a
perfect simulation of the random cluster measure using the results
of the previous sections. We explain these results in the case of
low temperatures. The case of high temperatures can be obtained in a
similar way.  In this section, from now on we assume that $G$ is a
simplicial graph, that $p_e>0$ for all $e\in E$ and that $
\lim_{\Lambda \uparrow E} \sup_{ e \not \in \Lambda} \hat{g}_e<1$.
The uniqueness proved in the following Corollary is well known at
least in the case of $\mathbb L^d$.

\begin{corollary}\label{cor:unicita}Assuming the hypotheses above
the random cluster measure on $G$ is unique.
\end{corollary}

\begin{proof}
Let  $\phi,\phi'$ be two DLR random cluster measures.

To prove that $\phi$ and $\phi'$ are equal we prove that for each
finite subset $F$ of $E$ the projections $\phi_F$ and $\phi'_F$ of
$\phi$ and $\phi'$ onto $\{0,1\}^F$ are equal. We denote also by
$X^{(N,\omega)}_F\in\{0,1\}^F$ the projection of
$X^{(N,\omega)}_{0,0}$.

Let $\omega$ be a random variable with law $\phi$ and $\omega'$ a
random variable with law $\phi'$.  By Theorem \ref{teorema1} we have
that
$$
\| X_F^{(N,\omega)} - X_F^{(N,\omega')} \|_{TV} \leq P\big(u\in
\calU\st N^w(u,F)\leq N \big)
$$
where the lefthandterm is the total variation distance between the
law of $X_F^{(N,\omega)}$ and $X_F^{(N,\omega')}$. Recall now that
as noticed in Section \ref{sec:dinamica} the DLR random cluster
measures are invariant under a $FK^N_{p,q}$-process. Hence, since
$\omega$ has law $\phi$, the random variable $X^{(N,\omega)}_F$ has
law $\phi_F$ and $X^{(N,\omega')}_F$ has law $\phi'_F$. Hence we get
$\| \phi_F - \phi'_F \|_{TV} \leq P\big(u\in \calU\st N^w(u,F)\leq N
\big)$. Finally by Theorem \ref{teo:pqbassa}, $C_F^w(u)$ is finite
for almost all $u\in \calU$. Hence $P(u\in \calU\st N^w(u,F)\leq N)$
goes to zero as $N$ goes to infinity. Hence $\phi$ and $\phi'$ are
equal.
\end{proof}

Now, under the same assumptions, given a finite subset $F$ of $E$ we
briefly describe an algorithm which furnishes a sampling of the
random cluster measure on $F$.

Let be given a generator of independent random numbers $u_{n,k}$.
The algorithm takes $F$ as an input, a suitable description of the
tessellation $\calP$, and gives as output subsets $C$, $\bar H$ of
$E^*$, a subset $\bar H_{-1}$ of $E$ and a configuration $Y\in
\{0,1\}^{\bar H_{-1}}$. It uses also local variables $D,D',F',L$ and
generates $u_{n,k}$ for $(n,e_k)\in \bar H$. We uses also the
notation of $B(u),M(u),W(u)$ for black, gray and white edges
introduced at the beginning of Section \ref{sec:bassa}. We describe
the algorithm with the following pseudocode.

\begin{enumerate}[\indent Step 1:]
\item generate the random numbers $u_{n,k}$ for $n=-1$ and $e_k\in
F$ and set $F'=\{-1\}\times F$;
\item set $D= F'\senza B(u))$;
\item set $L=D\cup \Delta_{\calP^*}(D)$ and generates the random
numbers $u_{n,k}$ for $(n,e_k)\in L\senza (D\cup F')$;
\item set $D'=D\cup (L \senza B(u))$ as in the construction of the
sets $D_i$ at the beginning of Section \ref{sec:bassa};
\item if $D'\neq D$ assigning to $D$ the value
given by $D'$ and goes to step 3;
\item if $D'=D$ then $C=D\cup F'$ and compute $N=N^{w}(u,F)$ as in Section \ref{sec:bassa};
\item use the formula \eqref{eq:H} to compute the sets $H^w_n$ and define
$\bar H=\bigcup_{n=N}^{-1} \{n\}\times \bar H_n$ where  $\bar
H_n=H^w_n\cup \Delta_\calP(H^w_n)$ as in Section \ref{sec:bassa};
\item generate the random numbers $u_{n,k}$ for $(n,e_k)\in \bar H\senza
C$;
\item $e\in \bar H_{-1}$ use the process described in Section
\ref{sec:dinamica} to compute the value $Y_e=X_{0,0}^{(\omega,N)}$
where $\omega_e=0$ for all $e\in \bar H_{N}$. As explained in the
proof of Theorem \ref{teorema1}, by Proposition \ref{prp:contorno2},
for this computation it is enough to know the value of $u$ only
inside the region $\bar H$.
\end{enumerate}

Notice that under our assumption, almost surely, this algorithm will
end in a finite number of steps. Notice also that the output $C$ is
the set $C^w_F(u,F)$, $\bar H$ is the set $\bar H^w(u,F)$. Finally
as explained in proof of uniqueness above $Y$ has law $\phi_F$ where
$\phi_F$ is the projection onto $\{0,1\}^F$ of the unique random
cluster measure on $\{0,1\}^E$.

In the case $\sup_{e\in E} \hat g_e<1$ it can be easily proved that
the average complexity of this algorithm goes linearly with the
cardinality of $F$.



\section{Applications to the Ising and Potts Model at low or high temperatures}

\label{sec:applicazioni}

We apply the results obtained in the previous sections for the
random cluster measure to the construction of a perfect simulation
of the Ising and of the Potts model with free boundary condition.

Let $G=(V,E)$ be a countable graph of finite degree. We briefly
recall how the Ising model with free boundary conditions can be
obtained from the random cluster measure with $q=2$, see the
original paper by Fortuin and Kasteleyn \cite{FK}, the book of
Grimmett \cite{Grimmett}, Chapter 1 or the book of Newman
\cite{Newman}, Chapter 3 for the details. Let $\grb \in (0,\infty)$
be the parameter of the Ising model called the inverse of the
temperature and let $J_e$ for $e\in E$ be the positive parameters
defining the Ising model called the interactions. In the random
cluster measure choose $q=2$ and $p_e=1-e^{-\grb J_e}$ for all $e\in
E$. Now we color, with $+$ or $-$, each vertex of $G$ using the
following rule. Consider the subgraph $G(\omega)=(V,E(\omega))$ of
$G$ where $\omega$ is selected using the random cluster measure
$\phi^0_{p,2}$ and $E(\omega)$ is defined as in Section
\ref{sec:misure}. Color the connected components of $G(\omega)$ with
$+$ or $-$ with an independent and uniform probability $1/2$.
Finally color each vertex with the same color of the connected
component containing the vertex. A similar construction can be used
to obtain the Potts model with free boundary conditions with the
only difference that there are $n$ colors, and we have to use the
random cluster measure $\phi^0_{p,n}$. Obviously the Ising model is
a particular case of the Potts model so that we will study this
second one in what follows.

\medskip

We explain how to obtain a perfect simulation of the Potts model
with free boundary condition in the case of low temperature. The
case of high temperature is completely analogous. So assume that the
hypotheses of Theorems~\ref{teorema1} and \ref{teo:pqbassa} are
satisfied: $G$ is a simplicial graph associated to the tessellation
$\calP$, $p_e>0$ for all $e\in E$ and $ \lim_{\Lambda \uparrow E}
\sup_{ e \not \in \Lambda} \hat{g}_e<1$. Given a finite subset $W$
of $V$ we briefly describe an algorithm which furnishes a sampling
of the Potts model on $W$.

Let $F'$ be the set of edges having at least one end vertex in $W$
and let $F$ be a connected set of edges containing $F'$. Now we use
the algorithm explained in the previous section to obtain a sampling
of the random cluster measure. In particular given $u$ we produce a
set $H = H^w_{-1}(u,F)$ and a configuration $Y=X_{0,0}^{(\omega,N)}$
(see Theorem \ref{teorema1} for the explanation of this notation).
Define also $K$ as the connected component of $H$ containing $F$ and
$\bar K$ as $K\cup \Delta_\calP(K)$. Let $V_{\bar K}$ be the set of
end vertexes of the edges in $\bar K$. Let also $GK(Y)$ be the
subgraph of $G(Y)$ having as vertexes the set $V_{\bar K}$ and as
edges the set $\bar K \cap E(Y)$. Notice that, as in the proof of
Theorem \ref{teorema1}, if $e \in \Delta_\calP(H)$ then $Y_e=1$.
Hence, by Proposition \ref{prp:contorno2}, if $x,y\in V_{\bar K}$
then they are in the same connected component of the graph $G(Y)$ if
and only if they are in same connected component of the graph
$GK(Y)$ which under our assumption, by Theorem \ref{teo:pqbassa} is
almost surely a finite graph. We determine the connected components
of $GK(Y)$ and we color each connected component as prescribed by
the Potts model. Finally we color each vertex in $W$ with the same
color of the connected component containing it.

It is easy to translate this description of this algorithm in an
actual pseudocode as we have done in the previous section.




\bibliographystyle{alpha}

\end{document}